\documentclass[11 pt]{amsart}
\usepackage{fancyhdr}
\usepackage[margin=1.15in]{geometry}
\usepackage{amsthm}
\usepackage{amsmath}
\usepackage{amssymb}
\usepackage{enumerate}
\usepackage{graphicx}
\usepackage{cancel}
\usepackage{hyperref}
\usepackage{comment}
\usepackage{tikz}
\usetikzlibrary{matrix}

\newtheorem{thm}{Theorem}[section]

\theoremstyle{definition}
\newtheorem{lem}[thm]{Lemma}

\newtheorem{exmp}[thm]{Example}
\theoremstyle{remark}
\newtheorem{lem*}{Lemma}

 {\begin{list}{}%
         {\setlength{\leftmargin}{#1}}%
         \item[]%
 }
 {\end{list}}
\input xy
\xyoption{all}

\title{Properties of Full-Flag Johnson Graphs}
\author[Dai]{Irving Dai}
\address[Irving Dai]{Harvard College\\ University Hall \\
Cambridge, MA 02138, USA}
\email{ifdai@college.harvard.edu}

\author[Greenberg]{Michael Greenberg}
\address[Michael Greenberg]{University of North Carolina, Chapel Hill}

\author[Schoem]{Noah Schoem}
\address[Noah Schoem]{University of Chicago}

\author[Tanzer]{Matt Tanzer}
\address[Matt Tanzer]{University of Michigan, Ann Arbor}

\begin{document}
\maketitle
\vspace{-1 cm}
\begin{abstract}  
We study a variant of the family of Johnson graphs, the Full-Flag Johnson graphs. We show that Full-Flag Johnson graphs are Cayley graphs on $S_n$ generated by certain classes of permutations, and that they are in fact generalizations of permutahedra. We derive some results about the adjacency matrices of Full-Flag Johnson graphs and apply these to the set of permutahedra to deduce part of their spectra.
\end{abstract}

\section{Introduction}
\noindent 
The Johnson graphs and permutahedra are well-known and well-studied families of graphs. For positive integers $n$ and $k$ with $k < n$, the Johnson graph $J(n,k)$ has vertex set given by the collection of all $k$-element subsets of $[n] = \{1, 2, \ldots, n\}$. Two vertices are adjacent if and only if their intersection has size $k-1$. The Johnson graphs are known to be Ramanujan, and their spectra are given by the Eberlein polynomials \cite{krebs}. For a positive integer $n$, the permutahedron of order $n$ has vertex set consisting of all permutations of $(n) = (1, 2, \ldots, n)$. Two vertices are adjacent if and only if they are of the form $(u_1, u_2, \ldots, u_i, u_{i+1}, \ldots, u_n)$ and $(u_1, u_2, \ldots, u_{i+1}, u_i, \ldots, u_n)$, respectively, that is, $u$ and $v$ are adjacent if they differ by a permutation that transposes two consecutive elements -- called a neighboring transposition. Permutahedra appear frequently in geometric combinatorics \cite{thomas}, and are known to be Hamiltonian \cite{elhashash}. \\
\\
In this paper, we present and discuss some characteristics of a variant of the set of Johnson graphs, the Full-Flag Johnson graphs. We show that Full-Flag Johnson graphs are Cayley graphs on $S_n$ generated by certain classes of permutations, and that they are in fact generalizations of permutahedra. We then investigate simple graph-theoretic properties of Full-Flag Johnson graphs, including connectedness and diameter. Finally, we derive some results about the adjacency matrices of Full-Flag Johnson graphs and apply these to the set of permutahedra in order to deduce part of their spectra. \\
\\
Much of this work was done at the Program in Mathematics for Young Scientists (Boston University, Summer 2010). The authors of this paper would like to thank Dr. Paul Gunnells and Nakul Dawra for their support and mathematical advice throughout this project. The results of Section 6 (Applications to Permutahedra) were largely developed by the first author after Summer 2010.

\section{Definitions and Examples}
\noindent
Let $n$ be a positive integer. A full-flag of subsets of $[n] = \{1, 2, \ldots, n\}$ is a sequence $U = (U_1, U_2, \ldots, U_n)$ such that
\begin{enumerate} \itemsep0pt
\item For each $i \in [n]$, $U_i$ is a subset of $[n]$,
\item For each $i \in [n-1]$, $U_i$ is a proper subset of $U_{i+1}$, and 
\item For each $i \in [n]$, $\left| U_i \right| = i$. 
\end{enumerate}
For example, one full-flag of subsets of $\{1, 2, 3, 4\}$ is $(\{3\}, \{3, 1\}, \{3, 1, 2\}, \{3, 1, 2, 4\})$. Let $k$ be a non-negative integer with $k < n$. The Full-Flag Johnson graph \textit{FJ}$(n,k)$ has vertex set $V(\textit{FJ}(n,k))$ given by the collection of all possible full-flags of $[n]$. Two vertices $U = (U_1, U_2, \ldots, U_n)$ and $V = (V_1, V_2, \ldots, V_n)$ are adjacent in \textit{FJ}$(n,k)$ if and only if $U_i \neq V_i$ for exactly $k$ integers $i \in [n]$. Equivalently, if we view $U$ and $V$ as collections of subsets of $[n]$, then $U$ and $V$ are adjacent if and only if $|U \cap V| = n - k$.\\
\\
We give an equivalent definition of \textit{FJ}$(n,k)$ that simplifies the vertex set at the expense of complicating relations between vertices. Let $U = (U_1, U_2, \ldots, U_n)$ be any vertex in \textit{FJ}$(n,k)$. For $1 < i \leq n$, the difference $U_i - U_{i-1}$ is a singleton set whose element is denoted by $u_i$. Letting $u_1$ be the singleton element of $U_1$, we may identify $U$ uniquely with the sequence $u = (u_1, u_2, \ldots, u_n)$. It is clear that $u$ must be a permutation of $(n) = (1, 2, \ldots, n)$ and that $U_i = \{u_1, u_2, \ldots, u_i\}$. Since every permutation of $(n)$ corresponds to a full-flag of subsets of $[n]$ in this manner, we may view the vertex set of \textit{FJ}$(n,k)$ as the collection of all permutations of $(n)$. Two vertices $u=(u_1, u_2,  \ldots, u_n)$ and $v=(v_1, v_2, \ldots, v_n)$ are adjacent if and only if there exist exactly $k$ integers $i \in [n]$ such that $\{u_1, u_2, \ldots, u_i\}$ and $\{v_1, v_2, \ldots, v_i\}$ are not equal. \\
\\
For example, two vertices in \textit{FJ}$(5,2)$ are $u=(1, 2, 3, 4, 5)$ and $v=(2, 1, 3, 5, 4)$; $u$ and $v$ are adjacent since $\{u_1, u_2, \ldots, u_i\} \neq \{v_1, v_2, \ldots, v_i\}$ for exactly two values of $i \in [5]$ ($i=1$ and $i=4$). On the other hand, suppose that $v$ were $(3, 2, 4, 1, 5)$. Then $\{u_1, u_2, \ldots, u_i\} \neq \{v_1, v_2, \ldots, v_i\}$ for exactly three values of $i \in [5]$ ($i = 1$, $i = 2$ and $i = 3$), so $u$ and $v$ would be adjacent in \textit{FJ}$(5,3)$ but not in \textit{FJ}$(5,2)$. \\
\\
In keeping with this new notation, for a permutation $u=(u_1, u_2, \ldots, u_n)$, denote by $u(i)$ the set $\{u_1, u_2, \ldots, u_i\}$ and the empty set $\emptyset$ for $i \in [n]$ and $i = 0$, respectively. Some algebraic rules for these sets are easily established. Since the elements of the sequence $u$ are distinct, for all $x$ and $y$ such that $0 \leq x < y \leq n$, we have $u(y)-u(x)=\{u_{x+1}, u_{x+2}, \ldots, u_{y-1}, u_y\}$. Thus, if $u(x)=v(x)$ and $u(y)=v(y)$, then clearly $u(y)-u(x)=v(y)-v(x)$. Similarly, it is also easily seen that if $u(x)=v(x)$ but $u(y)\neq v(y)$, then $u(y)-u(x) \neq v(y)-v(x)$. 

\begin{exmp}[Trivial Full-Flag Johnson Graph]
\end{exmp}
\noindent
As an initial example, consider the case when $k=0$. For every positive integer $n$, the vertices of \textit{FJ}$(n,0)$ are permutations of $(n)$, and two vertices $u=(u_1, u_2,  \ldots, u_n)$ and $v=(v_1, v_2, \ldots, v_n)$ are adjacent if and only if $u(i) \neq v(i)$ for no integers $i \in [n]$; that is, $u(i) = v(i)$ for every $i \in [n]$. It is easy to see that this implies $u=v$. Indeed, for arbitrary $i \in [n]$, we know that $u(i-1)=v(i-1)$ and $u(i)=v(i)$. Hence $u_i=v_i$. As this holds for all $i \in [n]$, we have $u=v$. Thus:

\begin{lem}
The graph  \textit{FJ}$(n,0)$ is the isolated graph, in which each vertex is adjacent only to itself. 
\end{lem}

\noindent
We use the adjective ``non-trivial" to describe Full-Flag Johnson graphs \textit{FJ}$(n,k)$ with $k > 0$.

\begin{exmp}[Permutahedron]
\end{exmp}
 \noindent
Consider the case when $k=1$. Two vertices $u$ and $v$ are adjacent in \textit{FJ}$(n,1)$ if and only if $u(i) \neq v(i)$ for exactly one $i \in [n]$. This implies that $u(x) = v(x)$ for each positive integer $x \neq i$, and in particular every positive integer $x < i$. By the argument of Example 2.1, we then have $u_1=v_1, u_2=v_2, \ldots, u_{i-1}=v_{i-1}$. Furthermore, $u(i-1)=v(i-1)$ but $u(i) \neq v(i)$, implying $u_i \neq v_i$. Thus, the first $i-1$ elements of the sequence $u$ are equal to the first $i-1$ elements of the sequence $v$, respectively, and the $i$th elements of $u$ and $v$ differ. \\
\\
Now, it can not be that $i=n$, since $u(n) = v(n) = \{1,2, \ldots, n\}$. Hence $i < n$, and $u(i+1) = v(i+1)$. Since $u(i-1) = v(i-1)$, we then have $\{u_i, u_{i+1}\}=\{v_i, v_{i+1}\}$. But $u_i \neq v_i$, so it must be that $v_i=u_{i+1}$ and $v_{i+1}=u_i$. \\
\\
Finally, we know that $u(x) = v(x)$ for all $x > i$. Again by the argument of Example 2.1, $u_{i+1}=v_{i+1}, u_{i+2}=v_{i+2}, \ldots, u_{n}=v_{n}$. Thus $u$ and $v$ are of the form $(u_1, u_2,  \ldots, u_i, u_{i+1}, \ldots, u_n)$ and $(u_1, u_2,  \ldots, u_{i+1}, u_i, \ldots, u_n)$, respectively; that is, they are related by a neighboring transposition. Conversely, it is easily seen that two vertices related by a neighboring transposition are indeed adjacent in \textit{FJ}$(n,1)$. Hence:

\begin{lem}
Two vertices in \textit{FJ}$(n,1)$ are adjacent if and only if they are related by a neighboring transposition. 
\end{lem}

\noindent
We thus see that for each positive integer $n$, the Full-Flag Johnson graph \textit{FJ}$(n,1)$ is in fact the permutahedron of order $n$. An example of $FJ(n,1)$ is given in Figure 1. \\

\begin{figure}[h!]
\centering
\includegraphics[scale=0.4]{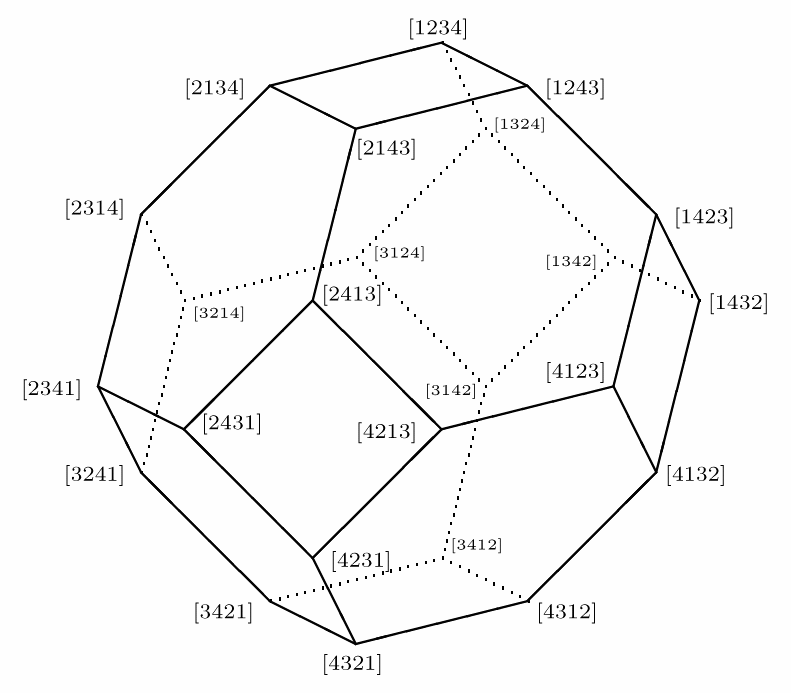}
\caption{The Full-Flag Johnson graph \textit{FJ}$(4,1)$.}
\end{figure}

\noindent
We close this section with a basic property of Full-Flag Johnson graphs.

\begin{thm} All non-trivial Full-Flag Johnson graphs are connected.
\end{thm}
\begin{proof}
Let $n$ and $k$ be positive integers with $k < n$, so that \textit{FJ}$(n,k)$ is a non-trivial Full-Flag Johnson graph. For every two vertices $u$ and $v$, we show that there exists a path between $u$ and $v$ in \textit{FJ}$(n,k)$. Clearly, it suffices to show this property for vertices related by a neighboring transposition, since given two arbitrary permutations $u$ and $v$ we may form a sequence of permutations beginning with $u$ and ending with $v$ such that each pair of consecutive permutations is related by a neighboring transposition. \\
\\
By re-labeling elements, let $u=(1, 2, \ldots, x, x+1,\ldots, n)$ and $v=(1, 2, \ldots, x+1, x, \ldots, n)$, so that $v$ is related to $u$ via a transposition of the elements $x$ and $x+1$ where $1 \leq x < n$. We proceed via casework on $x$ and $k$. By Lemma 2.4, if $k=1$ then $u$ and $v$ are trivially connected (being adjacent). We thus assume that $k > 1$. Then: \\
\\
Case 1: When $x+k \leq n$: Write $u = (1, 2, \ldots, x, x+1, \ldots, x + k, \ldots, n)$, where $x$, $x+1$, and $x+k$ are distinct positive integers (since $k > 1$). Consider the sequence of three vertices: 
\begin{align*}
u &= (1, 2, \ldots, x, x+1, \ldots, x + k, \ldots, n), \\
w &= (1, 2, \ldots, x + k, x, \ldots, x + 1, \ldots, n)\text{, and} \\
v &= (1, 2, \ldots, x+1, x, \ldots, x + k, \ldots, n),
\end{align*}
where only elements $x$, $x+1$, and $x+k$ have changed position among $u$, $w$, and $v$ (i.e., $u_i=w_i=v_i=i$ for all $i \notin \{x, x+1, x+k\}$). We claim that $u$ and $w$ are adjacent in \textit{FJ}$(n,k)$, that is, we assert that $u(i) \neq w(i)$ for exactly $k$ integers $i \in [n]$. \\
\\
Clearly, $u(i) = w(i)=\{1, 2, \ldots, i\}$ for all positive integers $i < x$. For all $i$ such that $x \leq i < x + k$, we have $u(i) \neq w(i)$ since $x + k \in w(i)$ but $x + k \notin u(i)$. Furthermore, for all $i$ such that $x+k \leq i$, again $u(i) = w(i)=\{1, 2, \ldots, i\}$. Hence $u(i) \neq w(i)$ for exactly those $i$ such that $x \leq i < x + k$, of which there are $k$. Thus $u$ and $w$ are adjacent in \textit{FJ}$(n,k)$. It can similarly be seen that $w$ and $v$ are adjacent, whence $u$ and $v$ are connected, as desired. \\
\\
Case 2: When $n < x+k$: Since $1 \leq n - k < x$, we may write $u = (1, 2, \ldots, n-k, \ldots, x, x+1, \ldots, n)$ for distinct positive integers $n-k$, $x$, and $x+1$. Now, either $x+1 = n$ or $x+1 \neq n$. If $x+1 = n$, consider the sequence of three vertices: 
\begin{align*}
u &= (1, 2, \ldots, n-k, \ldots, x, n), \\
w &= (1, 2, \ldots, x, \ldots, n, n-k)\text{, and} \\
v &= (1, 2, \ldots, n-k, \ldots, n, x),
\end{align*}
where only elements $n-k$, $x$, and $x+1=n$ have changed position among $u$, $w$, and $v$. If $x+1 \neq n$, consider:
\begin{align*}
u &= (1, 2, \ldots, n-k, \ldots, x, x+1, \ldots, n), \\
w &= (1, 2, \ldots, n, \ldots, x+1, x, \ldots, n-k)\text{, and} \\
v &= (1, 2, \ldots, n-k, \ldots, x+1, x, \ldots, n),
\end{align*}
where only elements $n-k$, $x$, $x+1$, and $n$ have changed position. A similar argument to the one presented in Case 1 shows that $u$ is adjacent to $w$ and that $w$ is adjacent to $v$, whence $u$ is connected to $v$, as desired. \\
\\
Having shown that every two permutations that differ by a neighboring transposition are connected, we conclude our proof.
\end{proof}

\section{Combinatorial Interpretation}
\noindent
In Lemma 2.4, we showed that for $k=1$ the adjacency relation between two vertices $u$ and $v$ in \textit{FJ}$(n, 1)$ could be stated in terms of $v$ being a particular permutation (specifically, a neighboring transposition) of $u$. It is natural to ask whether the adjacency between $u$ and $v$ can be similarly formulated for $u$ and $v$ in \textit{FJ}$(n,k)$ for $k > 1$. That is, let $u=(u_1, u_2,  \ldots, u_n)$ and $v=(v_1, v_2, \ldots, v_n)$ be adjacent in \textit{FJ}$(n,k)$. We then ask: how are $u$ and $v$ related in terms of one being a permutation of the other? \\
\\
We know that $u(i) \neq v(i)$ for exactly $k$ indices $i \in [n]$. Take the complement of this set of indices in $[n]$; that is, consider the set of positive indices $i$ such that $u(i) = v(i)$. It is clear that this has cardinality $n-k$; enumerate these $n-k$ integers in ascending order and consider an arbitrary pair of successive integers in this sequence, say $x$ and $y$. Since $u(x) = v(x)$ and $u(y) = v(y)$, we then have $u(y)-u(x) = v(y) - v(x)$, that is, $\{u_{x+1}, u_{x+2}, \ldots, u_{y}\} = \{v_{x+1}, v_{x+2}, \ldots, v_{y}\}$. Thus, the subsequence of elements $(v_{x+1}, v_{x+2}, \ldots, v_{y})$ from $v$ is a permutation of the subsequence of elements $(u_{x+1}, u_{x+2}, \ldots, u_{y})$ from $u$. \\
\\
Now we ask: is it possible that there is an integer $i$ with $x+1 \leq i < y$ such that the subsequence $(v_{x+1}, v_{x+2}, \ldots, v_i)$ is a permutation of the subsequence $(u_{x+1}, u_{x+2}, \ldots, u_i)$? We assert that the answer is ``no". For if there were, it is easily seen that we would have $u(i) = v(i)$, contradicting the fact that $x$ and $y$ are successive integers in our ordering of such indices. We thus say that $(v_{x+1}, v_{x+2}, \ldots, v_{y})$ is an irreducible permutation of $(u_{x+1}, u_{x+2}, \ldots, u_{y})$; that is, there is no integer $i$ with $x+1 \leq i < y$ such that the subsequences $(u_{x+1}, u_{x+2}, \ldots, u_i)$ and $(v_{x+1}, v_{x+2}, \ldots, v_i)$ are permutations of each other. \\
\\
We have thus established the following. Let $u=(u_1, u_2,  \ldots, u_n)$ and $v=(v_1, v_2, \ldots, v_n)$ be two adjacent vertices in \textit{FJ}$(n,k)$. Then there are exactly $n-k$ indices $i$, enumerated in ascending order as $x < y < z < \ldots < n$, such that $u(i)=v(i)$. Noting that the last integer in this ordering is $n$ (since $u(n)=v(n)$ for any $u$ and $v$), we may partition the index set $\{1, 2, \ldots, n\}$ into $\{1, 2, \ldots, x\} \cup \{x+1, x+2, \ldots, y\} \cup \{y+1, y+2, \ldots, z\} \cup \cdots$ so that the subsequences of $u$ and $v$ corresponding to any one partition are irreducible permutations of each other. That is, contiguous $(u_1, u_2, \ldots, u_x)$ and $(v_1, v_2, \ldots, v_x)$, $(u_{x+1}, u_{x+2}, \ldots, u_y)$ and $(v_{x+1}, v_{x+2}, \ldots, v_y)$, $(u_{y+1}, u_{y+2}, \ldots, u_z)$ and $(v_{y+1}, v_{y+2}, \ldots, v_z)$, $\ldots$ are irreducible permutations of each other. \\
\\
In general, if $v$ is a permutation of $u$ that can be decomposed into $N$ irreducible permutations in this manner, we say that $v$ is an $N$-reducible permutation of $u$. In our case, $v$ is an $(n-k)$-reducible permutation of $u$. Conversely, it is easily seen that if $v$ is an $(n-k)$-reducible permutation of $u$, then $u$ and $v$ are indeed adjacent in \textit{FJ}$(n,k)$. Hence:
\begin{lem}
Two vertices in \textit{FJ}$(n,k)$ are adjacent if and only if one is an $(n-k)$-reducible permutation of the other.
\end{lem}

\noindent
As an illustration of the above lemma, consider the vertices $u = (1, 2, 3, 4, 5, 6, 7)$ and $v = (2, 3, 1, 4, 6, 7, 5)$. It is easily seen that $u$ and $v$ are adjacent in \textit{FJ}$(7,4)$, with $u(i) \neq v(i)$ for $i \in \{1, 2, 5, 6\}$. The  complement of this set in $[7]$ is $\{3, 4, 7\}$. Adding spaces for clarity,
\begin{center}
$u = ( \ \ 1, 2, 3, \ \ 4, \ \ 5, 6, 7 \ \ )$ \\
$v = ( \ \ 2, 3, 1, \ \ 4, \ \ 6, 7, 5 \ \ )$
\end{center}
we see that the 1st through 3rd elements of $u$ and $v$ are irreducible permutations of each other, as are the (3+1)st through 4th (i.e., the fourth element), and the (4+1)st through 7th. \\
\\
We now introduce a slightly different way to view irreducible permutations. A reducible permutation matrix is a permutation matrix that can be decomposed into non-empty submatrices along its main diagonal that are also permutation matrices.  For example, the following $3 \times 3$ permutation matrix is reducible:
\[\left[ \begin{array}{ccc}
0 & 1 & 0 \\
1 & 0 & 0 \\
0 & 0 & 1 \end{array} \right]\]
since it may be decomposed into the $2 \times 2$ permutation matrix in the upper-left corner and the $1 \times 1$ permutation matrix in the lower-right corner. As another example, the $n \times n$ identity matrix can be decomposed into $n$ singleton submatrices along its main diagonal. A permutation matrix that is not reducible is said to be irreducible. An example of such a matrix is:
\[ \left[ \begin{array}{ccc}
0 & 0 & 1 \\
1 & 0 & 0 \\
0 & 1 & 0 \end{array} \right]\]
and it is clear that such matrices may be identified with the irreducible permutations discussed earlier. \\
\\
Similarly, for a positive integer $N$, an $N$-block diagonal permutation matrix is a permutation matrix that can be decomposed into exactly $N$ irreducible permutation submatrices along its main diagonal. The first example we gave above is two-block diagonal; the $n \times n$ identity matrix is $n$-block diagonal. Any irreducible permutation matrix is one-block diagonal (i.e., it can not be decomposed further). It is clear that each $N$-block diagonal permutation matrix corresponds to an $N$-reducible permutation, as discussed earlier (and vice-versa). We thus have:
\begin{thm} The Full-Flag Johnson graph \textit{FJ}$(n,k)$ is the Cayley graph on $S_n$ generated by the set of all $n \times n$ $(n-k)$-block diagonal permutation matrices.
\end{thm}
\begin{proof}
The theorem follows immediately from Lemma 3.1. 
\end{proof}

\noindent
Lemmas 2.2 and 2.4 follow immediately from Theorem 3.2, since the only $n \times n$ $n$-block diagonal permutation matrix is the $n \times n$ identity matrix, and every $n \times n$ $(n-1)$-block diagonal permutation matrix corresponds to a neighboring transposition. 

\section{Diameters of Full-Flag Johnson Graphs}
\noindent
We now give some results on the diameters of Full-Flag Johnson graphs, starting with the case when $k = 1$.  

\begin{thm} The diameter of \textit{FJ}$(n,1)$ is $n \choose 2$ for all $n \geq 2$.
\end{thm}
\begin{proof}
Let $n \geq 2$ and consider the Full-Flag Johnson graph \textit{FJ}$(n,1)$. Choose two arbitrary vertices $u$ and $v$. Without loss of generality, assume that  $u = (u_1, u_2, \ldots, u_n)$ and $v = (1, 2, \ldots, n)$. We show that there exists a path from $u$ to $v$ in \textit{FJ}$(n,1)$ of at most $n \choose 2$ edges. Indeed, let $u_i$ be the element of $u$ that is equal to 1, where $i \in [n]$, so that $u = (u_1, u_2, \ldots, u_{i-2}, u_{i-1}, 1, \ldots, u_n)$. Now consider:
\begin{align*}
&(u_1, u_2, \ldots, u_{i-2}, u_{i-1}, 1, \ldots, u_n), \\
&(u_1, u_2, \ldots, u_{i-2}, 1, u_{i-1}, \ldots, u_n), \\
&(u_1, u_2, \ldots, 1, u_{i-2}, u_{i-1}, \ldots, u_n), \\
&\ \ \ \ \ \ \ \ \ \ \ \ \ \ \ \ \ \ \ \ \ \ \vdots \\
&(1, u_1, \ldots, u_{i-3}, u_{i-2}, u_{i-1}, \ldots, u_n),
\end{align*}
that is, the sequence of vertices (beginning with $u$) in which the element 1 is gradually moved to the position of index one via successive neighboring transpositions to the left. This forms a path, since each pair of consecutive vertices are adjacent by Lemma 2.4. Since $i \leq n$, this path has at most $n-1$ edges. \\
\\
Now move the element 2 in the end vertex of this path to the position of index two via a similar sequence of successive neighboring transpositions. Since we are moving an element of maximal index $n$ to the position of index two, at most $n-2$ transpositions are needed. Thus, this new extra portion of the path has at most $n-2$ edges. Continuing on in this manner, we move the element 3 to the position of index three, which takes at most $n-3$ edges, and so on until we reach $v = (1, 2, \ldots, n)$. This takes at most $(n-1) + (n-2) + \cdots + 1 =$  $ n \choose 2$ edges in total. \\
\\
We now show that there exist two vertices $u$ and $v$ for which every joining path contains at least $n \choose 2$ edges. Let $u = (n, n-1, \ldots, 1)$ and $v = (1, 2, \ldots, n)$. Our assertion is easily proven using the disorder function $f$. Given a permutation $w = (w_1, w_2, \ldots, w_n)$ of $(n)$, define $f(w)$ to be the number of index-pairs $(i, j)$ with $1 \leq i < j \leq n$ such that $u_i > u_j$. In this case, $f(u) =$ $n \choose 2$ and $f(v) = 0$. But applying a neighboring transposition to $u$ clearly changes the value of $f$ by $\pm 1$, so at least $n \choose 2$ neighboring transpositions are needed. Thus every path from $u$ to $v$ has at least $n \choose 2$ edges. 
\end{proof}
\noindent
We now consider the opposite extreme when $k = n-1$. The $n = 2$ case is then trivial, so we assume $n \geq 3$.
\begin{thm} The diameter of \textit{FJ}$(n,n-1)$ is $2$ for all $n \geq 3$.
\end{thm}
\begin{proof}
Let $n \geq 3$ and consider the Full-Flag Johnson graph \textit{FJ}$(n,n-1)$. Choose two arbitrary vertices, $u$ and $v$. Without loss of generality, assume that $u = (u_1, u_2, \ldots, u_n)$ and $v = (1, 2, \ldots, n)$. We show that there exists a path from $u$ to $v$ in \textit{FJ}$(n, n-1)$ of at most two edges. Consider the possible values of $u_n$. Assume that $u_n = 1$. Then for all integers $i$ with $1 \leq i < n$, we have $1 \in v(i)$ but $1 \notin u(i)$. Since trivially $u(n) = v(n)$, this shows that $u(i) \neq v(i)$ for exactly $n-1$ values of $i \in [n]$, so $u$ and $v$ are adjacent in \textit{FJ}$(n,n-1)$. Now assume that $u_n > 1$. Consider the permutation:
\[
w = (u_n, u_n+1, \ldots, n, 1, 2, \ldots, u_n-1)
\]
formed by concatenating the arithmetic sequences $(u_n, u_n+1, \ldots, n)$ and $(1, 2, \ldots, u_n-1)$. (Note that if $u_n = n$, the first sequence consists of a single element.) Clearly $u$ is adjacent to $w$, since for all integers $i$ with $1 \leq i < n$ we have $u_n \in w(i)$ but $u_n \notin u(i)$. We also claim that $w$ is adjacent to $v$; that is, $w(i) \neq v(i)$ for all positive integers $i < n$. For every positive integer $i < u_n$, we have $u_n \in w(i)$ but $u_n \notin v(i)$ (since $v(i) = \{1, 2, \ldots, i\}$ and $i < u_n$), so $w(i) \neq v(i)$. On the other hand, for every positive integer $i$ such that $u_n \leq i < n$, we have $u_n-1 \in v(i)$ but $u_n-1 \notin w(i)$, so again $w(i) \neq v(i)$. Thus $w$ and $v$ are adjacent in \textit{FJ}$(n,n-1)$, as desired. This shows the sequence of vertices $(u, w, v)$ is a path of length two that connects $u$ and $v$. \\
\\
We now present two vertices $u$ and $v$ for which every joining path has at least two edges. Let $u = (1, 2, \ldots, n)$ and $v = (2, 1, \ldots, n)$. Clearly, $u(i) \neq v(i)$ for exactly one value of $i \in [n]$ (i.e., $i = 1$). Since $n-1 > 1$, $u$ and $v$ are not adjacent in \textit{FJ}$(n,n-1)$, and so every path connecting them must have at least two edges. 
\end{proof} 

\noindent
To obtain results on the diameter of \textit{FJ}$(n,k)$ in the general case for $0 \leq k < n$, we consider how adjacent vertices are related as permutations of one another. Let $u=(u_1, u_2,  \ldots, u_n)$ and $v=(v_1, v_2, \ldots, v_n)$ be two adjacent vertices in \textit{FJ}$(n,k)$. Then there are exactly $k$ indices $i \in [n]$ such that $u(i) \neq v(i)$. Denote this set of indices by $I$. \\
\\
Consider the extreme case when all of the elements of $I$ are consecutive, i.e., $I = \{x, x +1, x+2, \ldots, x+k-1\}$ for some $x \in [n]$ with $x + k -1 < n$. Then we have $u(i) = v(i)$ for all $i \in [n]$ with $i \leq x - 1$ or $x + k \leq i$. In particular, $u(x-1) = v(x-1)$ and $u(x+k) = v(x+k)$, so the subsequences $(u_{x}, u_{x+1}, \ldots, u_{x+k})$ of $u$ and $(v_{x}, v_{x+1}, \ldots, v_{x+k})$ of $v$ are permutations of each other. Furthermore, the argument of Example 2.1 shows that $u_i = v_i$ for all $i \in [n] - I$, so $u$ and $v$ differ by a permutation of $k+1$ consecutive elements. As in the proof of Theorem~4.1, this permutation can be decomposed into an application of at most $k+1 \choose 2$ neighboring transpositions. \\
\\
Generally, not all of the elements of $I$ are consecutive. However, we clam that $u$ and $v$ still differ by at most $\binom{k+1}{2}$ neighboring transpositions. Partition $I = I_1 \cup I_2 \cup \cdots \cup I_m$ such that, for every $j \in [m]$, $I_j$ is a maximal subset of $I$ consisting of consecutive indices of $I$. Applying an argument to the index set $I$ as in the above extreme case, the vertices $u$ and $v$ differ by at most $\binom{|I_1|+1}{2} + \binom{|I_2|+1}{2} + \cdots + \binom{|I_m|+1}{2}$ neighboring transpositions. By noting that $\sum_{j=1}^m |I_j| = |I| = k$, we can derive the inequality $\sum_{j=1}^m \binom{|I_j| + 1}{2} \leq \binom{|I| + 1}{2} = \binom{k+1}{2}$. We thus have:

\begin{lem} Every two vertices that are adjacent in \textit{FJ}$(n,k)$ differ by at most $k+1 \choose 2$ neighboring transpositions.
\end{lem}

\noindent
We use this result to prove a lower bound on the diameter of Full-Flag Johnson graphs in the general case.

\begin{thm} The diameter of every non-trivial Full-Flag Johnson graph \textit{FJ}$(n,k)$ is bounded below by $\binom{n}{2}/\binom{k+1}{2}$.
\end{thm}

\begin{proof}
Let $n$ and $k$ be positive integers with $k < n$ so that \textit{FJ}$(n,k)$ is a non-trivial Full-Flag Johnson graph. To establish our bound, we exhibit two vertices of \textit{FJ}$(n,k)$ for which every joining path has at least $\binom{n}{2}/\binom{k+1}{2}$ edges. Consider two arbitrary vertices $u$ and $v$, and let $P$ be a path from $u$ to $v$ of length $|P|$. By Lemma 4.3, every pair of adjacent vertices in \textit{FJ}$(n,k)$ differ by at most $k + 1 \choose 2$ neighboring transpositions, so $u$ and $v$ differ by an application of at most $|P|$ $k + 1 \choose 2$ neighboring transpositions. Now, for every path $P$ joining the two vertices $u = (1, 2, \ldots, n)$ and $v = (n, n-1, \ldots, 1)$, an upper bound on the number of neighboring transpositions by which $u$ and $v$ differ is $|P|$ $k + 1 \choose 2$, and a lower bound of $\binom{n}{2}$ was shown in the proof of Theorem~4.1. Thus $|P| \geq \binom{n}{2}/\binom{k+1}{2}$, as desired.
\end{proof}

\section{Recursive Structures}
\noindent
We now investigate the recursive structure of the graphs \textit{FJ}$(n,k)$ for successive values of $n$. \\
\\
Let $n$ be a positive integer. For every integer $i \in [n+1]$, define the $i$th insertion function on $S_n$ to be the mapping $\phi_i$ that takes each element of $S_n$ and inserts an $n+1$ in the $i$th place. That is, for every permutation $u = (u_1, u_2, \ldots, u_{i-1}, u_i, \ldots, u_n)$ of $(n)$, let $\phi_i(u)$ be the permutation $(u_1, u_2, \ldots, u_{i-1}, n+1, u_i, \ldots, u_n)$ of $(n+1)$. For every Full-Flag Johnson graph \textit{FJ}$(n,k)$, it is clear that that, for every $i \in [n+1]$, each $\phi_i$ constitutes a bijection between $V(\textit{FJ}(n,k))$ and a subset of $V(\textit{FJ}(n+1,k))$. In fact, the following theorem holds:

\pagebreak
\begin{thm} The graph \textit{FJ}$(n,k)$ is a subgraph of \textit{FJ}$(n+1,k)$.
\end{thm}

\begin{proof} Let \textit{FJ}$(n,k)$ be a Full-Flag Johnson graph. We claim that $\phi_1$ induces a graph isomorphism between \textit{FJ}$(n,k)$ and a subgraph of \textit{FJ}$(n+1,k)$. We have already noted that $\phi_1$ is a bijection between $V(\textit{FJ}(n,k))$ and a subset of $V(\textit{FJ}(n+1,k))$; it remains to show that $\phi_1$ preserves adjacency and non-adjacency. Let $u$ and $v$ be two arbitrary vertices in \textit{FJ}$(n,k)$. We show that $u$ and $v$ are adjacent in \textit{FJ}$(n,k)$ if and only if $\phi_1(u)$ and $\phi_1(v)$ are adjacent in \textit{FJ}$(n+1,k)$. Write $u= (u_1, u_2, \ldots, u_n)$ and $v = (v_1, v_2, \ldots, v_n)$, then:
\begin{align*}
\phi_1(u) &= (n+1, u_1, u_2, \ldots, u_n)\text{ and} \\
\phi_1(v) &= (n+1, v_1, v_2, \ldots, v_n).
\end{align*}
For every $i \in [n]$, we claim that $u(i) = v(i)$ if and only if $\phi_1(u)(i+1) = \phi_1(v)(i+1)$. This follows immediately from the equalities:
\begin{align*}
\phi_1(u)(i + 1)& = \{n+1, u_1, u_2, \ldots, u_i \} = u(i) \cup \{n+1\}\text{ and} \\ 
\phi_1(v)(i + 1) &= \{n+1, v_1, v_2, \ldots, v_i \} = v(i) \cup \{n+1\}
\end{align*}
and the fact that $n+1$ is not an element of either $u$ or $v$. Thus $u(i) \neq v(i)$ for exactly $k$ indices $i \in [n]$ if and only if $\phi_1(u)(i + 1) \neq \phi_1(v)(i + 1)$ for exactly those $k$ indices $i \in [n]$. Noting that $\phi_1(u)(1) = \{n+1\} = \phi_1(v)(1)$, it is then clear that $\phi_1$ preserves adjacency and non-adjacency, as desired. \\
\\
It can similarly be seen that $\phi_{n+1}$ is a graph isomorphism between \textit{FJ}$(n,k)$ and a subgraph of \textit{FJ}$(n+1,k)$. However, in general for $2 \leq i \leq n$, $\phi_i$ is not a graph isomorphism.
\end{proof}

\noindent
Generalizing this theorem requires a few additional notions. An ordering of a set is a particular permutation of its elements. A vertex ordering of a graph is an ordering of its vertex set. \\
\\
Let $G$ be a graph of order $n$ and fix some ordering $S = (v_1, v_2, \ldots, v_n)$ of its vertex set. The adjacency matrix of $G$ with respect to $S$ is the $n \times n$ matrix, denoted by $A(G, S)$, whose $(i,j)$th element is 1 if $v_i$ and $v_j$ are adjacent in $G$ and 0 otherwise. \\
\\
Let \textit{FJ}$(n,k)$ be a Full-Flag Johnson graph and $S$ be some ordering of its vertex set. For every $i \in [n+1]$, define $\phi_i(S)$ to be the sequence of permutations formed by taking the (permutation-)elements of $S$ and applying $\phi_i$ to each one of them. For example, let $n = 3$ and
\begin{center}
$S = (123, 132, 213, 231, 312, 321)$,
\end{center}
where we have suppressed extra commas and parentheses for clarity (e.g., (1, 2, 3) is written as 123). Then:
\begin{align*}
\phi_1(S) &= (4123, 4132, 4213, 4231, 4312, 4321), \\
\phi_2(S) &= (1423, 1432, 2413, 2431, 3412, 3421), \\
\phi_3(S) &= (1243, 1342, 2143, 2341, 3142, 3241)\text{, and} \\
\phi_4(S) &= (1234, 1324, 2134, 2314, 3124, 3214).
\end{align*}
Now let $\overline{S}$ be the concatenation of $\phi_1(S), \phi_2(S), \ldots, \phi_{n+1}(S)$, in that order. It is clear that $\overline{S}$ is a vertex ordering of \textit{FJ}$(n+1,k)$. Our goal will be to study the adjacency matrix $A(\textit{FJ}(n+1,k),\overline{S})$ of \textit{FJ}$(n+1,k)$ in terms of the adjacency matrices $A(\textit{FJ}(n,k),S)$ of \textit{FJ}$(n,k)$ and other lower-order Full-Flag Johnson graphs. \\
\\
The dimensions of $A(\textit{FJ}(n,k),S)$ and $A(\textit{FJ}(n+1,k),\overline{S})$ are $n! \times n!$ and $(n+1)! \times (n+1)!$, respectively. It is thus natural to decompose $A(\textit{FJ}(n+1,k),\overline{S})$ into $(n+1)^2$ $n! \times n!$ submatrices, each identical in size to $A(\textit{FJ}(n,k),S)$. For every pair of integers $i, j \in [n+1]$, denote by $A(\textit{FJ}(n+1,k),\overline{S})[i,j]$ the $(i, j)$th submatrix of $A(\textit{FJ}(n+1,k),\overline{S})$, as illustrated in Figure 2. It is clear that $A(\textit{FJ}(n+1,k),\overline{S})[i,j]$ consists of the entries in $A(\textit{FJ}(n+1,k),\overline{S})$ corresponding to adjacencies between vertices in $\phi_i(S)$ and vertices in $\phi_j(S)$. \\

\begin{figure}[h!]
\centering
\includegraphics[scale=0.7]{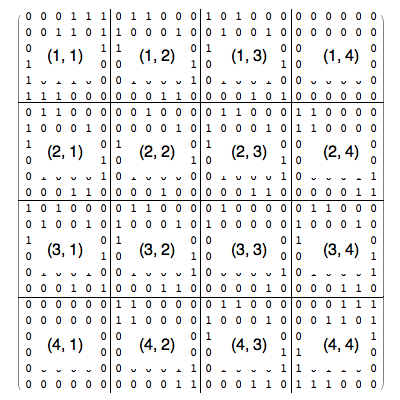}
\caption{The adjacency matrix of \textit{FJ}$(4,2)$ divided into $4^2$ $3! \times 3!$ submatrices.}
\end{figure}

\begin{thm} Let \textit{FJ}$(n,k)$ be a non-trivial Full-Flag Johnson graph and $S$ be an ordering of its vertex set. With respect to the vertex ordering $\overline{S}$, which is the concatenation of $\phi_1(S), \phi_2(S), \ldots, \phi_{n+1}(S)$ in that order, $\textit{FJ}(n+1, k)$ satisfies the following properties:
\begin{enumerate} \itemsep0pt
\item For all $i, j \in [n + 1]$ with $\left|i-j\right|>k$, $A(\textit{FJ}(n+1, k),\overline{S})[i,j]$ is the $n! \times n!$ zero matrix,
\item Relating $\textit{FJ}(n+1,k)$ with $\textit{FJ}(n,k)$: \\
$A(\textit{FJ}(n+1,k),\overline{S})[1,1] = A(\textit{FJ}(n+1,k),\overline{S})[n+1,n+1] = A(\textit{FJ}(n,k),S)$, and
\item Relating $\textit{FJ}(n+1,k)$ with $\textit{FJ}(n,k-1)$: for all $i, j \in [n+1]$ with $\left|i-j\right|=1$, \\
$A(\textit{FJ}(n+1,k),\overline{S})[i,j] = A(\textit{FJ}(n, k-1),S)$. 
\end{enumerate}

\end{thm}
\noindent
Before we give a proof of Theorem 5.2, it will be illuminating to consider the example given in Figure 3. Let $n = 3$, $k = 2$, and fix the vertex ordering $S = (123, 132, 213, 231, 312, 321)$. 
\pagebreak
\begin{figure}[h!]
\centering
\includegraphics[scale=0.8]{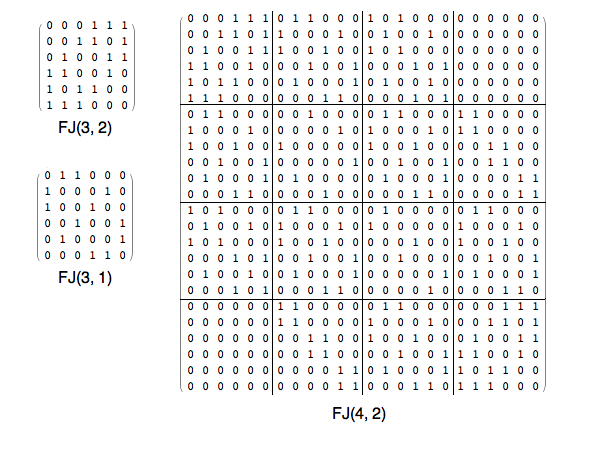}
\caption{An illustration of Theorem 5.2 for $S = (123, 132, 213, 231, 312, 321)$.}
\end{figure}
\noindent
\\
Figure 3 makes the assertions of Theorem 5.2 readily apparent. Part $(1)$ of Theorem 5.2 states that the upper right and lower left $3! \times 3!$ submatrices of $A(\textit{FJ}(4,2),\overline{S})$ are zero. Part $(2)$ states that the upper left and lower right $3! \times 3!$ submatrices of $A(\text{FJ}(4,2),\overline{S})$ should be identical to $A(\textit{FJ}(3,2), S)$, which we see is indeed the case. Finally, part $(3)$ states that the $3! \times 3!$ submatrices on either side of the main diagonal are identical to $A(\textit{FJ}(3,1),S)$. \\
\\
We now give a proof of Theorem 5.2. 

\begin{proof} We begin with the first assertion. Consider arbitrary $i, j \in [n+1]$ with $\left|i-j\right| > k$, and assume without loss of generality that $i < j$. The entries of $A(\textit{FJ}(n+1,k),\overline{S})[i,j]$ are the entries of $A(\textit{FJ}(n+1,k),\overline{S})$ corresponding to adjacencies between the vertices of $\phi_i(S)$ and the vertices of $\phi_j(S)$. We thus need to show that no vertex of $\phi_i(S)$ is adjacent to a vertex of $\phi_j(S)$ in \textit{FJ}$(n+1, k)$. Let $u$ and $v$ be vertices in $\phi_i(S)$ and $\phi_j(S)$, respectively. Then $u$ has $n+1$ in its $i$th place, and $v$ has $n+1$ in its $j$th place. This means that $u(x) \neq v(x)$ for all $x \in [n+1]$ with $i \leq x < j$, since $n+1 \in u(x)$  but $n+1 \notin v(x)$. We thus have at least $j-i$ integers $x \in [n+1]$ for which $u(x) \neq v(x)$. But by assumption $j-i > k$, so $u$ and $v$ can not be adjacent in \textit{FJ}$(n+1,k)$. This shows that $A(\textit{FJ}(n+1,k),\overline{S})[i,j]$ is the zero matrix. \\
\\
The second assertion follows immediately from the graph isomorphisms $\phi_1$ and $\phi_{n+1}$ given in the proof of Theorem 5.1. Note that the submatrix $A(\textit{FJ}(n+1,k),\overline{S})[1,1]$ consists of the entries of $A(\textit{FJ}(n+1,k),\overline{S})$ corresponding to the adjacencies among the vertices of $\phi_1(S)$. But $\phi_1(S)$ is constructed by taking $S$ and applying the graph isomorphism $\phi_1$ to each element. Hence $A(\textit{FJ}(n+1,k),\overline{S})[1,1]$ is identical to $A(\textit{FJ}(n,k),S)$. Similarly, $\phi_{n+1}(S)$ is the image of the elements of $S$ under $\phi_{n+1}$, and so $A(\textit{FJ}(n+1,k),\overline{S})[n+1,n+1]$ is also equal to $A(\textit{FJ}(n,k),S)$. \\
\\
We now prove the third assertion. Let $i \in [n]$. We show that $A(\textit{FJ}(n+1,k),\overline{S})[i,i+1]$ is equal to the adjacency matrix $A(\textit{FJ}(n, k-1),S)$ of $\textit{FJ}(n, k-1)$. By symmetry, this will also show that $A(\textit{FJ}(n+1,k),\overline{S})[i+1,i] = A(\textit{FJ}(n,k-1),S)$, and so complete the proof. We prove that for every two vertices $u$ and $v$ in \textit{FJ}$(n,k-1)$, $u$ and $v$ are adjacent in \textit{FJ}$(n,k-1)$ if and only if $\phi_i(u)$ and $\phi_{i+1}(v)$ are adjacent in \textit{FJ}$(n+1,k)$. Indeed, let:
\begin{align*}
u &= (u_1, u_2, \ldots, u_{i-1}, u_i, u_{i+1}, \ldots, u_n)\text{ and} \\
v &= (v_1, v_2, \ldots, v_{i-1}, v_i, v_{i+1}, \ldots, v_n)
\end{align*}
so that:
\begin{align*}
&\phi_i(u) =  (u_1, u_2, \ldots, u_{i-1}, n+1, u_i, u_{i+1}, \ldots, u_n)\text{ and} \\
&\phi_{i+1}(v) = (v_1, v_2, \ldots, v_{i-1}, v_i, n+1, v_{i+1}, \ldots, v_n).
\end{align*}
Then the following assertions hold: 
\begin{enumerate} \itemsep0pt
\item For all integers $x$ with $1 \leq x \leq i-1$, $u(x) = v(x)$ if and only if $\phi_i(u)(x) = \phi_{i+1}(v)(x)$, since $\phi_i(u)(x) = u(x)$ and $\phi_i(v)(x) = v(x)$, and 
\item For all integers $x$ with $i \leq x \leq n$, $u(x) = v(x)$ if and only if $\phi_i(u)(x+1) = \phi_{i+1}(v)(x+1)$, since $\phi_i(u)(x+1) = u(x) \cup \{n+1\}$ and $\phi_{i+1}(v)(x+1) = v(x) \cup \{n+1\}$.
\end{enumerate}
Let $f: [n] \rightarrow [n+1] - \{i\}$ be the bijection defined by

\begin{center}
$f(x)=\left\{     
\begin{array}{lr}
x & : 1\leq x \leq i-1, \\
x+1 & : i \leq x \leq n.
\end{array}
\right.$ \\
\end{center}

\noindent
Then $u(x) \neq v(x)$ for $k-1$ indices $x \in [n]$ if and only if $\phi_i(u)(f(x)) \neq \phi_{i+1}(v)(f(x))$ for exactly those $k-1$ indices $x \in [n]$. In addition, we also have $\phi_i(u)(i) \neq \phi_{i+1}(v)(i)$, since $n+1 \in \phi_i(u)(i)$ but $n+1 \notin \phi_{i+1}(v)(i)$. Hence $u$ and $v$ are adjacent in \textit{FJ}$(n,k-1)$ if and only if $\phi_i(u)$ and $\phi_{i+1}(v)$ are adjacent in \textit{FJ}$(n+1,k)$, as desired.
\end{proof}

\section{Applications to Permutahedra}
\noindent
As shown in Lemma 2.4, for each positive integer $n$ the Full-Flag Johnson graph \textit{FJ}$(n,1)$ may be identified with the order-$n$ permutahedron, which is the Cayley graph on $S_n$ generated by the set of all neighboring transpositions. A neighboring transposition in a permutation of $S_n$ that interchanges the $i$th and $(i+1)$st elements for some $i \in [n-1]$ is termed as the $(i, i+1)$-transposition. Hence \textit{FJ}$(n,1)$ as a Cayley graph is generated by the set of all $(i, i+1)$-transpositions with $i \in [n-1]$ and has regularity $n-1$. \\
\\
We consider Theorem 5.2 in the case that $k = 1$. An illustration of the $n=3$ case is given in Figure 4.

\pagebreak

\begin{figure}[h!]
\includegraphics[scale = 0.6]{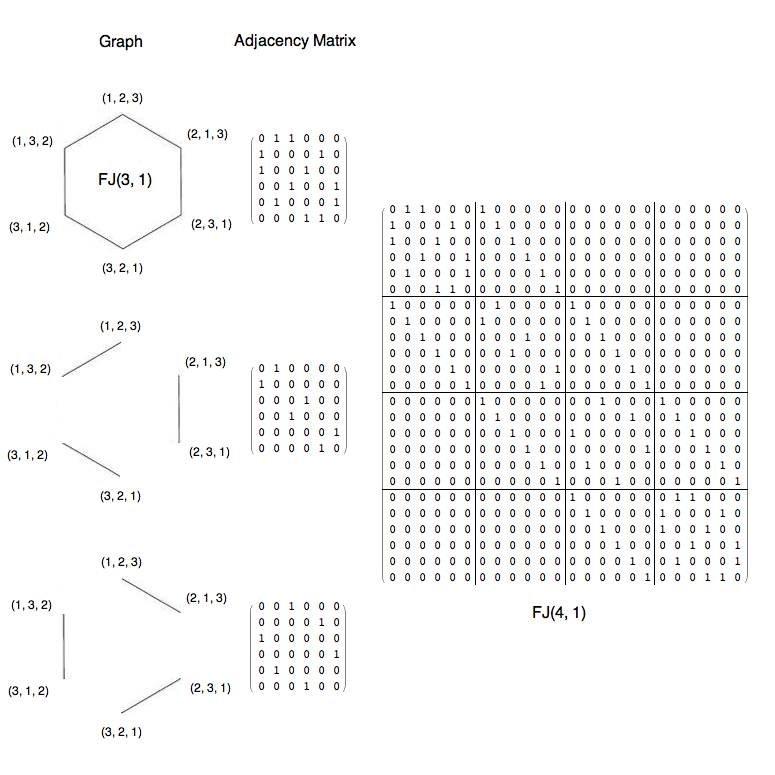}
\caption{An illustration of Theorem 6.1 for $S = (123, 132, 213, 231, 312, 321)$.}
\end{figure}
\noindent
The $3! \times 3!$ submatrices $A(\textit{FJ}(4,1),\overline{S})[i,j]$ for all $i, j \in [4]$ with $|i - j| > 1$ are identically zero, while the $3! \times 3!$ submatrices flanking the main diagonal are the identity. (As shown in Lemma 2.2, the adjacency matrix of \textit{FJ}$(n, 0)$ is the $n! \times n!$ identity matrix.) We thus consider the $3! \times 3!$ submatrices along the main diagonal. As in Theorem 5.2, the upper left and lower right $3! \times 3!$ submatrices are identical to the adjacency matrix $A(\textit{FJ}(3,1),S)$. If we compute the adjacency matrices with respect to $S$ of the two proper, non-trivial, regular subgraphs of \textit{FJ}$(3, 1)$, we furthermore see that these are precisely the remaining two $3! \times 3!$ submatrices on the main diagonal. This is the content of the following theorem.

\begin{thm}
 Let \textit{FJ}$(n,1)$ be the order-$n$ permutahedron and $S$ be an ordering of its vertex set. With respect to the vertex ordering $\overline{S}$, which is the concatenation of $\phi_1(S), \phi_2(S), \ldots, \phi_{n+1}(S)$ in that order, $\textit{FJ}(n+1, 1)$ satisfies the following properties:
\begin{enumerate} \itemsep0pt
\item The matrix $A(\textit{FJ}(n+1,1),\overline{S})$ is block tri-diagonal, with identity matrices on the flanking diagonals. That is, for all $i, j \in [n+1]$: if $|i -j| = 1$, then $A(\textit{FJ}(n+1,1),\overline{S})[i,j]$ is the identity matrix, and if $|i - j| > 1$ then $A(\textit{FJ}(n+1,1),\overline{S})[i,j]$ is the zero matrix, and 
\item For every $i \in [n+1]$, $A(\textit{FJ}(n+1,1),\overline{S})[i,i]$ is identical to the adjacency matrix of a regular subgraph of \textit{FJ}$(n,1)$ with respect to $S$: if $i = 1$ or $i = n+1$, then this subgraph is \textit{FJ}$(n,1)$ itself $($and has regularity $n-1$$)$, else, it is the subgraph generated by excluding the $(i-1,i)$-transposition from the generating set of \textit{FJ}$(n,1)$ $($and has regularity $n-2$$)$.
\end{enumerate}

\end{thm}

\begin{proof} The first assertion follows immediately from Theorem 5.2. We thus prove the second assertion. Let $i \in[n+1]$ be arbitrary. The case when $i$ is $1$ or $n+1$ again follow from Theorem 5.2; thus, assume that $1 < i < n+1$. We prove that $A(\textit{FJ}(n+1,1),\overline{S})[i,i]$ is identical to the adjacency matrix with respect to $S$ of the Cayley graph on $S_n$ generated by the set of all neighboring transpositions excluding the $(i-1,i)$-transposition. \\
\\
Let $u$ and $v$ be two arbitrary permutations in $S_n$. We show that $\phi_i(u)$ and $\phi_i(v)$ are adjacent in \textit{FJ}$(n+1,1)$ if and only if $u$ and $v$ are related by a neighboring transposition other than the $(i-1,i)$-transposition. Write:
\begin{align*}
u &= (u_1, u_2, \ldots, u_{i-1}, u_i, \ldots, u_n)\text{ and} \\
v &= (v_1, v_2, \ldots,  v_{i-1}, v_i, \ldots, v_n)
\end{align*}
so that:
\begin{align*}
\phi_i(u) &= (u_1, u_2, \ldots, u_{i-1}, n+1, u_i, \ldots, u_n)\text{ and} \\
\phi_i(v) &= (v_1, v_2, \ldots,  v_{i-1}, n+1, v_i, \ldots, v_n).
\end{align*}
Assume that $\phi_i(u)$ and $\phi_i(v)$ are adjacent in \textit{FJ}$(n+1,1)$. Then $\phi_i(u)$ and $\phi_i(v)$ are related via a neighboring transposition. Now, it is clear that this neighboring transposition can not interchange the elements $u_{i-1}$ and $n+1$ of $u$. For if it did, we would have $v_{i-1} = n+1$, which is impossible. Similarly, it can not interchange the elements $n+1$ and $u_i$ of $u$. Hence the neighboring transposition in question must interchange the elements $u_x$ and $u_{x+1}$ of $u$ for some $x \in [n] - \{i-1\}$. This means that $v$ is related to $u$ via the $(x, x+1)$-transposition; since $x \neq i-1$, we have $(x, x+1) \neq (i-1, i)$, as desired. \\
\\
Conversely, it is easily seen that if $u$ and $v$ are related by a neighboring transposition other than the $(i-1,i)$-transposition, then $\phi_i(u)$ and $\phi_i(v)$ must be adjacent in \textit{FJ}$(n+1,1)$. This completes the proof.
\end{proof}
\noindent
By Theorem 6.1, for all $i, j \in [n+1]$ each submatrix $A(\textit{FJ}(n+1,1),\overline{S})[i,j]$ of $A(\textit{FJ}(n+1,1),\overline{S})$ is identical to the adjacency matrix of a regular graph. For all $n \geq 2$ and orderings $S$ of $S_n$, define the regularity matrix of \textit{FJ}$(n, 1)$ with respect to $S$ to be the $n \times n$ matrix whose $(i,j)$th element is the regularity of $A(FJ(n, 1), S)[i,j]$ for all $i, j \in [n]$. Denoting this matrix by $M$, Theorem 6.1 shows that $M$ is of the form: \\
\\
\[
M =
\left[ {\begin{array}{ccccccccc}
n-2 & 1 & 0 & 0 & \ldots & 0 & 0 & 0 & 0 \\
1 & n-3 & 1 & 0 & \ldots & 0 & 0 & 0 & 0 \\
0 & 1 & n-3 & 1 & \ldots & 0 & 0 & 0 & 0 \\
\ & \ & \ & \ & \vdots \\
0 & 0 & 0 & 0 & \ldots & 1 & n-3 & 1 & 0 \\
0 & 0 & 0 & 0 & \ldots & 0 & 1 & n-3 & 1 \\
0 & 0 & 0 & 0 & \ldots & 0 & 0 & 1 & n-2
 \end{array} } \right].
\] 
\\
\begin{thm}
The eigenvalue spectrum of $M$ is a subset of the eigenvalue spectrum of \textit{FJ}$(n, 1)$ (not considering multiplicity).
\end{thm}
\begin{proof}
Denote by $\mathbb{R}$ the set of all reals. Let $s_i$ be the vector in $\mathbb{R}^{n!}$ with a 1 in positions $(i-1)(n-1)! + k$ for all $k \in [n-1]$ and 0s elsewhere. For example, if $n = 3$, 
\[
s_1 = \left[\begin{array}{c} 1 \\1 \\0 \\0 \\0 \\0 \\ \end{array}\right], \ \
s_2 = \left[\begin{array}{c} 0 \\0 \\1 \\1 \\0 \\0 \\ \end{array}\right], \text{ and} \ \
s_3 = \left[\begin{array}{c} 0 \\0 \\0 \\0 \\1 \\1 \\ \end{array}\right].
\] 
Consider the vector space homomorphism $\phi: \mathbb{R}^n \rightarrow \mathbb{R}^{n!}$ mapping the canonical basis vector $e_i$ to $s_i$ for every $i \in [n]$. Denote by $A$ the adjacency matrix $A(\textit{FJ}(n,1),S)$. Then it is not difficult to see that for every $v \in \mathbb{R}^n$, $\phi(Mv) = A\phi(v)$. It suffices to prove the equivalence for all basis vectors $e_i$ where $i \in [n]$. \\
\\
In general, let $R$ be the adjacency matrix of a regular graph with order $m$ and regularity $r$. If $w$ is the column vector in $\mathbb{R}^m$ consisting of all $1$s, it is clear that $Rw = rw$. The desired equivalence then follows from writing the matrix multiplication $As_i$ in terms of blocks and applying Theorem 6.1. This immediately implies the theorem, since if $v$ is an eigenvector of $M$ with eigenvalue $\lambda$, then $\phi(v)$ is an eigenvector of $A$ with the same eigenvalue.
\end{proof}
\noindent
As an example of Theorem 6.2, consider the case when $n = 4$. The eigenvalue spectrum of \textit{FJ}$(4,1)$ is (without multiplicity):
\begin{center}
$\{3, 1+\sqrt{2}, \sqrt{3}, 1, -1+\sqrt{2}, 1-\sqrt{2}, -1, -\sqrt{3}, -1-\sqrt{2}, -3\}$
\end{center}
as can be verified through the use of a computer algebra system. The regularity matrix of \textit{FJ}$(4,1)$ is:
\[
M = \left[\begin{array}{cccc}
2 & 1 & 0 & 0 \\
1 & 1 & 1 & 0 \\
0 & 1 & 1 & 1 \\
0 & 0 & 1 & 2
\end{array}\right]
\]
and has the eigenvalue spectrum:
\begin{center}
$\{3, 1+\sqrt{2}, 1, 1-\sqrt{2}\}$,
\end{center}
which is a subset of the spectrum of \textit{FJ}$(4,1)$. In general, determining the eigenvalue spectrum of a graph is very difficult; in our case, \textit{FJ}$(n, 1)$ has $n!$ vertices, so brute-force computation is impractical for large values of $n$. Theorem 6.2 succeeds in reducing the computation to finding the eigenvalues of an $n \times n$ matrix, but only gives a subset of the spectrum. It can trivially be seen that this subset must include the largest eigenvalue of \textit{FJ}$(n, 1)$; we conjecture that it also includes the second-largest. (This has been verified for all $n \leq 5$.) For applications of the second-largest eigenvalue of a graph, see \cite{hoory}.

\end{document}